\documentclass[12pt, twoside]{article}
\usepackage{amsmath,amsthm,amssymb}
\usepackage{times}
\usepackage{enumerate}
\usepackage{bm}
\usepackage[all]{xy}
\usepackage{mathrsfs}
\usepackage{amscd}
\usepackage{color}
\def\titlerunning#1{\gdef\titrun{#1}}
\makeatletter
\def\author#1{\gdef\autrun{\def\and{\unskip, }#1}\gdef\@author{#1}}
\def\address#1{{\def\and{\\\hspace*{18pt}}\renewcommand{\thefootnote}{}%
\footnote {#1}}%
\markboth{\autrun}{\titrun}}
\makeatother
\def\email#1{e-mail: #1}
\def\subjclass#1{{\renewcommand{\thefootnote}{}%
\footnote{\emph{
Mathematics Subject Classification (2010):} #1}}
}
\def\keywords#1{\par\medskip
\noindent\textbf{Keywords.} #1}

\newtheorem{thm}{Theorem}[section]
\newtheorem{cor}[thm]{Corollary}
\newtheorem{lem}[thm]{Lemma}
\newtheorem{prop}[thm]{Proposition}
\theoremstyle{definition}
\newtheorem{defi}[thm]{Definition}

\numberwithin{equation}{section}

\input xy
\xyoption{all}

\def \C {\mathbb{C}}

\def \G {\mathcal{G}}

\def \D {\mathcal{D}}

\def \W {\mathcal{W}}

\def \a {\alpha }

\def \de {\delta}

\def \la {\lambda}
\def \La {\Lambda}
\def\w {\omega}
\def\Om{\Omega}

\def\pa{\partial}
\def\na {\nabla}
\def\Ga{\Gamma}
\def\ra{\rightarrow}

\begin{document}
\baselineskip=17pt

\titlerunning{A lower bound on the solutions of Kapustin-Witten equations}
\title{A lower bound on the solutions of Kapustin-Witten equations}

\author{Teng Huang}

\date{}

\maketitle

\address{T. Huang: Department of Mathematics, University of Science and Technology of China,
                   Hefei, Anhui 230026, PR China; \email{oula143@mail.ustc.edu.cn}}

\subjclass{53C07; 58E15}
\begin{abstract}
In this article,\ we consider the Kapustin-Witten equations on a closed $4$-manifold.\ We study certain analytic properties of solutions to the equations on a closed manifold.\ The main result is that there exists an $L^{2}$-lower bound on the extra fields over a closed four-manifold satisfying certain conditions if the connections are not ASD connections.\ Furthermore,\ we also obtain a similar result about the Vafa-Witten equations.
\end{abstract}

\keywords{Kapustin-Witten equations, ASD connections, flat $G_{\C}$-connections}

\section{Introduction}

Let $X$ be an oriented $4$-manifold with a given Riemannian metric $g$.\ We use the metric on $X$ to define the Hodge star operator on $\La^{\bullet}T^{\ast}X$ and then write the bundle of $2$-forms as the direct sum $\Om^{2}T^{\ast}X=\Om^{+}\oplus\Om^{-}$ with $\Om^{+}$ denoting the bundle of self-dual $2$-forms and with $\Om^{-}$ denoting the bundle of anti-self-dual $2$-forms,\ with respect to this Hodge star.\ If $\w$ denotes a given $2$-form,\ then its respective self-dual and anti-self-dual parts are denoted by $\w^{+}$ and $\w^{-}$.

Let $P$ be a principal bundle over $X$ with structure group $G$.\ Supposing that $A$ is the connection on $P$,\ then we denote by $F_{A}$ its curvature $2$-form,\ which is a $2$-form on $X$ with values in the bundle associated to $P$ with fiber the Lie algebra of $G$ denoted by $\mathfrak{g}$.\ We define by $d_{A}$ the exterior covariant derivative on section of $\La^{\bullet}T^{\ast}X\otimes(P\times_{G}\mathfrak{g})$ with respect to the connection $A$.

The Kapustin-Witten equations are defined on a Riemannian $4$-manifold given a principal bundle $P$.\ For most our considerations,\ $G$ is taken to be $SO(3)$.\ The equations require a pair $(A,\phi)$ of connection on $P$ and section of $T^{\ast}X\otimes(P\times_{G}\mathfrak{g})$ to satisfy
\begin{equation}\label{82}
(F_{A}-\phi\wedge\phi)^{+}=0\quad and\quad (d_{A}\phi)^{-}=0\quad and\quad d_{A}\ast\phi=0.
\end{equation}
Witten \cite{W1,W2,W3,W4},\ \cite{GW},\ \cite{KEW} and also Haydys \cite{Hay} proposed that certain linear combinations of the equations in (\ref{82}) and the version with the self and anti-self dual forms interchanged should also be considered.\ The latter are parametrized by $\tau\in[0,1]$ which can be written as
\begin{equation}\label{81}
\begin{split}
&\tau(F_{A}-\phi\wedge\phi)^{+}=(1-\tau)(d_{A}\phi)^{+},\\
&(1-\tau)(F_{A}-\phi\wedge\phi)^{-}=-\tau(d_{A}\phi)^{-},\\
&d_{A}\ast\phi=0.\\
\end{split}
\end{equation}
The $\tau=0$ version of (\ref{81}) and the $\tau=1$ version of (\ref{81}) is the version of (\ref{82}) that is defined on $X$ with its same metric but with its orientation reversed.\ In the case when $X$ is compact,\ Kapustin and Witten \cite{KEW} proved that the solution of (\ref{82}) with $\tau\in(0,1
)$ exists only in the case when $P\times_{G}\mathfrak{g}$ has zero first Pontrjagin number,\ and if so,\ the solutions are such that $A+\sqrt{-1}\phi$ defining a flat $PSL(2,\C)$ connection.\ A nice discussions of there equations can be found in \cite{GU}.

If $X$ is compact,\ and $(A,\phi)$ obeys (\ref{82}),\ then
\begin{equation}\nonumber
\begin{split}
YM_{\C}(A+\sqrt{-1}\phi)&=\int_{X}\big{(}|(F_{A}-\phi\wedge\phi)|^{2}+|d_{A}\phi|^{2}\big{)}dvol_{g}\\
&=\int_{X}\big{(}|F_{A}|^{2}-2\langle F_{A},\phi\wedge\phi\rangle+|\phi\wedge\phi|^{2}-\langle\ast[F_{A},\phi],\phi\rangle\big{)}dvol_{g}\\
&=\int_{X}\big{(}|F_{A}|^{2}-4\langle F^{+}_{A},\phi\wedge\phi\rangle+2|(\phi\wedge\phi)^{+}|^{2}\big{)}dvol_{g}\\
&=\int_{X}\big{(}|F_{A}|^{2}-2|F^{+}_{A}|\rangle\big{)}dvol_{g}\\
&=-4\pi^{2}p_{1}(P\times_{SO(3)}\mathfrak{s}\mathfrak{o}(3)).\\
\end{split}
\end{equation}
where $\mathfrak{s}\mathfrak{o}(3)$ is the Lie algebra of $SO(3)$ and $p_{1}(P)$ is the first Pontrjagin number,\ the second line we used the fact (\ref{41}).\ This identity implies,\ among other things,\ that there are no solutions to (\ref{82}) in the case when $X$ is compact and the first Pontrjagin number is positive.\ It also implies that $(A,\phi)$ solves (\ref{82}) when $X$ is compact and $p_{1}(P\times_{SO(3)}\mathfrak{s}\mathfrak{o}(3))=0$ if and only if $A+\sqrt{-1}\phi$ defines a flat $SL(2,\C)$ connection on $X$.

In \cite{T1},\ Taubes studied the Uhlenbeck style compactness problem for $SL(2,\C)$ connections,\ including solutions to the above equations,\ on four-manifolds (see also \cite{T2,T3}).\ In \cite{TY2},\ Tanaka observed that equations (\ref{82}) on a compact K\"{a}hler surface are the same as Simpson's equations,\ and proved that the singular set introduced by Taubes for the case of Simpson's equations has a structure of a holomorphic subvariety.

We define the configuration spaces
\begin{equation}\nonumber
\begin{split}
&\mathcal{C}:=\mathcal{A}_{P}\times\Om^{1}(X,\mathfrak{g}_{P}),\\
&\mathcal{C}':=\Om^{2,+}(X,\mathfrak{g}_{P})\times\Om^{2,-}(X,\mathfrak{g}_{P}).\\
\end{split}
\end{equation}
We also define the gauge-equivariant map
$$KW:\mathcal{C}\rightarrow\mathcal{C}',$$
$$KW(A,\phi)=\big{(}F_{A}^{+}-(\phi\wedge\phi)^{+},(d_{A}\phi)^{-}\big{)}.$$
Mimicking the setup of Donaldson theory,\ the $KW$-moduli space is
$$M_{KW}(P,g):=\{(A,\phi)\mid KW(A,\phi)=0\}/\mathcal{G}_{P}.$$
The moduli space $M_{ASD}$ of all ASD connections can be embedded into $M_{KW}$ via $A\mapsto(A,0)$,\ here $A$ is an ASD connection.\ In particular,\ Taubes \cite{T4,T5} proved the existence of ASD connections on certain four-manifolds satisfying extra conditions.\ For any positive real constant $C\in\mathbb{R}^{+}$,\ we define the $C$-truncated moduli space
$$M^{C}_{KW}(P,g):=\{(A,\phi)\in M_{KW}(P,g)\mid\|\phi\|_{L^{2}(X)}\leq C\}.$$

In this article,\ we assume that there is a peculiar circumstance in that one obtains an $L^{2}$ lower bound on the extra field $\phi$  on a closed,\ oriented,\ four-dimension manifold,\ $X$ satisfies certain conditions if the connection is not an ASD connection.
\begin{thm}(Main Theorem)\label{2}
Let $X$ be a closed,\ oriented,\ $4$-dimensional Riemannian manifold with Riemannaian metric $g$,\ let $P\rightarrow X$ be a principal $G$-bundle with $G$ being a compact Lie group  with $p_{1}(P)$ negative and be such that there exist $\mu,\de>0$ with the property that $\mu(A)\geq\mu$ for all $A\in\mathfrak{B}_{\de}(P,g)$,\ where $\mathfrak{B}_{\de}(P,g):=\{[A]:\|F^{+}_{A}\|_{L^{2}(X)}<\de\}$ and $\mu(A)$ is as in (\ref{12}).\ There exist a positive constant,\ $C$,\ with the following significance.\ If $(A,\phi)$ is an $L^{2}_{1}$ solution of $M^{C}_{KW}(P,g)$,\ then $A$ is anti-self-dual with respect to the metric $g$.
\end{thm}
\begin{cor}\label{92}
Let $X$ be a closed,\ oriented,\ $4$-dimension Riemannian manifold with Riemannaian metric $g$ ,\ let $P$ be a principal $SO(3)$ bundle with $P\times_{SO(3)}\mathfrak{s}\mathfrak{o}(3)$ has negative first Pontrjagin class over $X$.\ Then there is an open dense subset,\ $\mathscr{C}(X,p_{1}(P))$,\ of the Banach space,\ $\mathscr{C}(X)$,\ of conformal equivalence classes,\ $[g]$,\ of $C^{r}$ Riemannian metrics on $X$ (for some integer $r\geq3$) with the following significance.\ If $[g]\in\mathscr{C}(X,p_{1}(P))$,\ then there exist a positive constant,\ $C$,\ with the following significance.\ Suppose that $P$ and $X$ obeys one of the following sets of conditions:

(1)$b^{+}(X)=0$; or

(2)$b^{+}(X)>0$ and the second Stiefel-Whitney class,\ $w_{2}(P)\in H^{2}(X;\mathbb{Z}/2\mathbb{Z})$,\ is non-trivial.

If $(A,\phi)$ is an $L^{2}_{1}$ solution of $M^{C}_{KW}(P,g)$,\ then $A$ is anti-self-dual with respect to the metric $g$.
\end{cor}

The organization of this paper is as follows.\ In section 2,\ we first review some estimates of the Kapustin-Witten equations.\ Thanks to Uhlenbeck's work,\ we observe that $A$ must be a flat connection if $\|\phi\|_{L^{2}}$ is sufficiently small,\ when $A+\sqrt{-1}\phi$ is a flat $G_{\C}$-connection and $\phi\in\ker{d^{\ast}_{A}}$.\ In section 3,\ we study certain analytic properties of solutions to the Kapustin-Witten equations on the closed manifolds.\ In section 4,\ we generalize the previous observation to the case of Kapustin-Witten equations.\ More precisely ,\ we show that if $X$ satisfy certain conditions and $\|\phi\|_{L^{2}}$ is sufficiently small,\ there exists a ASD connection near given connection $A$ measured by $\|F^{+}_{A}\|_{L^{2}}$,\ then we complete the proof of main theorem by the similar way for the case of flat $G_{\C}$-connections.\ In the last section,\ we also obtain similar result about the Vafa-Witten equations.

\section{Fundamental preliminaries }

We shall generally adhere to the now standard gauge-theory conventions and notation of Donaldson and Kronheimer \cite{DK}.\ Throughout our article,\ $G$  denotes a compact Lie group and $P$ a smooth principal $G$-bundle over a compact Riemannnian manifold $X$ of dimension $n\geq2$ and endowed with Riemannian metric $g$.\ For $u\in L^{p}(X,\mathfrak{g}_{P})$,\ where $1\leq p<\infty$ and $k$ is an integer,\ we denote
\begin{equation}\label{Y4}
\|u\|_{L^{p}_{k,A}(X)}:=\big{(}\sum_{j=0}^{k}\int_{X}|\na^{j}_{A}u|^{p}dvol_{g}\big{)}^{1/p},
\end{equation}
where $\na_{A}:C^{\infty}(X,\Om^{\cdot}(\mathfrak{g}_{P}))\rightarrow\C^{\infty}(X,T^{\ast}X\otimes\Om^{\cdot}(\mathfrak{g}_{P}))$ is the covariant derivative induced by the connection,\ $A$,\ on $P$ and Levi-Civita connection defined by the Riemannian metric,\ $g$,\ on $T^{\ast}X$,\ and all associated vector bundle over $X$,\ and $\na^{j}_{A}:=\na_{A}\circ\ldots\circ\na_{A}$ (repeated $j$ times for $j\geq0$).\ The Banach spaces,\ $L^{p}_{k,A}(X,\Om^{l}(\mathfrak{g}_{P}))$,\ are the completions of $\Om^{l}(X,\mathfrak{g}_{P})$ with respect to the norms (\ref{Y4}).\ For $p=\infty$,\ we denote
\begin{equation}
\|u\|_{L^{\infty}_{k,A}(X)}:=\sum_{j=0}^{k}ess\sup_{X}|\na^{j}_{A}u|.
\end{equation}
For $p\in[1,\infty)$ and nonnegative integer $k$,\ Banach space duality to define
$$L^{p'}_{-k,A}(X,\Om^{l}(\mathfrak{g}_{P})):=\big{(}L^{p}_{k,A}(X,\Om^{l}(\mathfrak{g}_{P}))\big{)}^{\ast},$$
where $p'\in[1,\infty)$ is the dual exponent defined by $1/p+1/p'=1$.

\subsection{Identities for the solutions}

This section derives some basic identities that are obeyed by solutions to Kapustin-Witten equations.
\begin{thm}(Weitezenb\"{o}ck formula)
\begin{equation}\label{40}
\begin{split}
&d_{A}^{\ast}d_{A}+d_{A}d_{A}^{\ast}=\na^{\ast}_{A}\na_{A}+Ric(\cdot)+\ast[\ast F_{A},\cdot]\ on\ \Om^{1}(X,\mathfrak{g}_{P})\\
\end{split}
\end{equation}
where $Ric$ is the Ricci tensor.
\end{thm}
As a simple application of Weitezenb\"{o}ck formula,\ we have the following proposition,
\begin{prop}\label{200}
If $(A,\phi)$ is a solution of Kapustin-Witten equations,\ then
\begin{equation}\label{3}
\na_{A}^{\ast}\na_{A}\phi+Ric\circ\phi+2\ast\big{[}\ast(\phi\wedge\phi)^{+},\phi\big{]}=0.
\end{equation}
\end{prop}
\begin{proof}
From $(d_{A}\phi)^{-}=0$,\ we have $d_{A}\phi=\ast d_{A}\phi$.\ Then we have
$$d_{A}d_{A}\phi=d_{A}\ast d_{A}\phi.$$
Since $d_{A}^{\ast}=-\ast d_{A}\ast$,\ we obtain
\begin{equation}\label{41}
d^{\ast}_{A}d_{A}\phi=-\ast[F_{A},\phi].
\end{equation}
From (\ref{40}) and (\ref{41}),\ we completes the proof of Proposition \ref{200}.
\end{proof}

Using a technique of Taubes (\cite{T4}, p.166) also described in (\cite{BM},\ p.23--24),\ we combine the Weitzenb\"{o}ck formula with Morrey's mean-value inequality to deduce a bound on $\|\phi\|_{L^{\infty}}$ in terms of $\|\phi\|_{L^{2}}$.\ First,\ we recall a mean value inequality as follow:

\begin{thm}(\cite{BM} Theorem 3.1.2)\label{116}
Let $X$ be a smooth closed Rimeannian manifold.\ For all $\la>0$ there are constants $\{C_{\la}\}$ with the following property.\ Let $V\rightarrow X$ be any real vector bundle equipped with a metric,\ $A\in\mathcal{A}_{V}$ any smooth metric-compatible connection,\ and $\sigma\in\Om^{0}(X,V)$ any smooth section with satisfies the pointwise inequality
$$\langle\sigma,\na^{\ast}_{A}\na_{A}\sigma\rangle\leq\la|\sigma|^{2}.$$
Then $\sigma$ satisfies the estimate
$$\|\sigma\|_{L^{\infty}}\leq C_{\la}\|\sigma\|_{L^{2}}.$$
\end{thm}

\begin{thm}\label{16}
Let $X$ be a compact $4$-dimensional Riemannian manifold.\ There exists a constant,\ $C=C(X)$,\ with the following property.\ For any principal bundle $P\rightarrow X$ and any $L^{2}_{1}$ solution $(A,\phi)$ to the Kapustin-Witten equations,
\begin{equation}\nonumber
\|\phi\|_{L^{\infty}(X)}\leq C\|\phi\|_{L^{2}(X)}.
\end{equation}
\end{thm}
\begin{proof}
By Theorem \ref{42},\ we may assume that $(A,\phi)$ is smooth.\ Form (\ref{3}),\ in pointwise,
\begin{equation}\label{4}
\langle\na^{\ast}_{A}\na_{A}\phi,\phi\rangle=-{\color{red}\big{(}}\langle Ric\circ\phi,\phi\rangle+4|(\phi\wedge\phi)^{+}|^{2}{\color{red}\big{)}}.
\end{equation}
Since $X$ is compact,\ we get a pointwise bound of the form
\begin{equation}\label{115}
\langle\na_{A}^{\ast}\na_{A}\phi,\phi\rangle\leq\la|\phi|^{2}
\end{equation}
for a constant $\la$ depending on Riemannian curvature of $X$.\ From (\ref{115}) and Theorem \ref{116},\ we obtain
$$\|\phi\|_{L^{\infty}(X)}\leq C\|\phi\|_{L^{2}(X)}.$$
\end{proof}

\subsection{Flat $G_{\C}$-connections}

Let $P\rightarrow X$ be a principal $G$-bundle with $G$ being a compact Lie group with $p_{1}(P)$ is zero,\ then the solutions $(A,\phi)$ to the Kapustin-Witten equations are flat $G_{\C}$-connections:
\begin{equation}\nonumber
F_{A}-\phi\wedge\phi=0\quad and\quad d_{A}\phi=0\quad and\quad d_{A}\ast\phi=0.
\end{equation}
First,\ we review a key result due to Uhlenbeck for the connections with $L^{p}$-small curvature ($2p>n$)\cite{U2}.
\begin{thm}(\cite{U2} Corollary 4.3)\label{33}\label{83}\label{117}
Let $X$ be a closed,\ smooth manifold of dimension $n\geq2$ and endowed with a Riemannian metric,\ $g$,\ and $G$ be a compact Lie group,\ and $2p>n$.\ Then there are constants,\ $\varepsilon=\varepsilon(n,g,G,p)\in(0,1]$ and $C=C(n,g,G,p)\in[1,\infty)$,\ with the following significance.\ Let $A$ be a $L^{p}_{1}$ connection on a principal $G$-bundle $P$ over $X$.\ If
$$\|F_{A}\|_{L^{p}(X)}\leq\varepsilon,$$
then there exist a flat connection,\ $\Ga$,\ on $P$ and a gauge transformation $g\in L^{p}_{2}(X)$ such that

(1) $d^{\ast}_{\Ga}\big{(}g^{\ast}(A)-\Ga\big{)}=0\ on\ X,$

(2) $\|g^{\ast}(A)-\Ga\|_{L^{p}_{1,\Ga}}\leq C\|F_{A}\|_{L^{p}(X)}\ and$

(3) $\|g^{\ast}(A)-\Ga\|_{L^{n/2}_{1,\Ga}}\leq C\|F_{A}\|_{L^{n/2}(X)}.$

\end{thm}
\begin{thm}\label{109}
Let $X$ be a closed,\ oriented,\ $4$-dimensional Riemannian manifold with Riemannaian metric $g$,\ let $P\rightarrow X$ be a principal $G$-bundle with $G$ being a compact Lie group  with $p_{1}(P)=0$.\ There exist a positive constant,\ $C$,\ with the following significance.\ If $(A,\phi)$ is an $L^{2}_{1}$ solution of $M^{C}_{KW}(P,g)$,\ then $A$ is a flat connection.
\end{thm}

\begin{proof}
By Theorem \ref{42},\ we can assume $(A,\phi)$ is smooth.\ From the Theorem \ref{16},\ we have
$$\|\phi\|_{L^{\infty}(X)}\leq C_{7}\|\phi\|_{L^{2}(X)},$$
where $C_{7}$ is only dependent on manifold.\ Since $(A,\phi)$ is a solution of the Kapustin-Witten equations,\ we have
$$\|F_{A}\|_{L^{p}(X)}\leq\|\phi\wedge\phi\|_{L^{p}(X)}\leq C_{8}\|\phi\|^{2}_{L^{2}(X)},$$
where $p>2$.\ We can  choose $\|\phi\|_{L^{2}(X)}\leq C$ sufficiently small such that $C_{8}C^{2}\leq\varepsilon$,\ where $\varepsilon$ is the constant in the hypothesis of Theorem \ref{117}.\ Then from Theorem \ref{83},\ there exist a flat connection $\Ga$ such that
$$\|A-\Ga\|_{L^{2}_{1}(X)}\leq C_{9}\|F_{A}\|_{L^{2}(X)}.$$
Using the Weitezenb\"{o}ck formula,\ we have identities
\begin{equation}\label{85}
(d^{\ast}_{\Ga}d_{\Ga}+d_{\Ga}d^{\ast}_{\Ga})\phi=\na_{\Ga}^{\ast}\na_{\Ga}\phi+Ric\circ\phi,
\end{equation}
and
\begin{equation}\label{84}
(d^{\ast}_{A}d_{A}+d_{A}d^{\ast}_{A})\phi=\na_{A}^{\ast}\na_{A}\phi+Ric\circ\phi+\ast[\ast F_{A},\phi],
\end{equation}
which lead to  the following two integral inequalities
\begin{equation}\label{86}
\|\na_{\Ga}\phi\|^{2}_{L^{2}(X)}+\int_{X}\langle Ric\circ\phi,\phi\rangle\geq0.
\end{equation}
and
\begin{equation}\label{87}
\|\na_{A}\phi\|_{L^{2}(X)}^{2}+\int_{X}\langle Ric\circ\phi,\phi\rangle+2\|F_{A}\|^{2}=0.
\end{equation}
On the other hand,\ we also have another integral inequality
\begin{equation}\label{88}
\begin{split}
\|\na_{A}\phi-\na_{\Ga}\phi\|^{2}_{L^{2}(X)}&\leq\|[A-\Ga,\phi]\|^{2}_{L^{2}(X)}\\
&\leq C_{10}\|A-\Ga\|^{2}_{L^{4}(X)}\|\phi\|^{2}_{L^{4}(X)}\\
&\leq C_{11}\|F_{A}\|^{2}_{L^{2}(X)}\|\phi\|^{2}_{L^{2}(X)}.\\
\end{split}
\end{equation}
Combing these inequalities,\ we arrive at
\begin{equation}\nonumber
\begin{split}
0&\leq\|\na_{\Ga}\phi\|_{L^{2}(X)}^{2}+\int_{X}\langle Ric\circ\phi,\phi\rangle\\
&\leq\|\na_{A}\phi\|_{L^{2}(X)}^{2}+\int_{X}\langle Ric\circ\phi,\phi\rangle+\|\na_{A}\phi-\na_{\Ga}\phi\|_{L^{2}(X)}^{2}\\
&\leq(C_{12}\|\phi\|_{L^{2}(X)}^{2}-2)\|F_{A}\|_{L^{2}(X)}^{2}.\\
\end{split}
\end{equation}
We can choose $\|\phi\|_{L^{2}(X)}\leq C$ sufficiently small such that $C_{12}C^{2}\leq1$,\ then $F_{A}$ vanishes.
\end{proof}

\subsection{A vanishing theorem on extra fields}

As usual,\ we define the stabilizer group $\Ga_{A}$ of $A$ in the gauge group $\G_{P}$ by
$$\Ga_{A}:=\{g\in\G_{P}\mid g^{\ast}(A)=A\}.$$
\begin{defi}
A connection $A$ is said to be $irreducible$ if the stabilizer group $\Ga_{A}$ is isomorphic to the centre of $G$,\ and $A$ is called $reducible$ otherwise.
\end{defi}
\begin{lem}(\cite{DK} Lemma 4.3.21)
If $A$ is an irreducible $SU(2)$ or $SO(3)$ anti-self-dual connection on a bundle $E$ over a simply connected four-manifold $X$,\ then the restriction of $A$ to any non-empty open set in $X$ is also irreducible.
\end{lem}
\begin{thm}\label{80}
Let $X$ be a simply-connected Riemannian four-manifold,\ let $P\rightarrow X$ be an $SU(2)$ or $SO(3)$ principal bundle.\ If $A\in\mathcal{A}_{P}$ is irreducible anti-self-dual connection and $\phi\in\Om^{1}(X,\mathfrak{g}_{P})$ satisfy
\begin{equation}\nonumber
\phi\wedge\phi=0\quad and\quad d_{A}\phi=0\quad and\quad d^{\ast}_{A}\phi=0
\end{equation}
then $\phi=0$.
\end{thm}
\begin{proof}
Since $F^{+}_{A}=0$,\ $\phi\wedge\phi=0$,\ then $\phi$ has at most rank one.\ Let $Z^{c}$ denote the complement of the zero of $\phi$.\ By unique continuation of the elliptic equation $(d_{A}+d^{\ast}_{A})\phi=0$,\ $Z^{c}$ is either empty or dense.

The Lie algebra of $SU(2)$ or $SO(3)$ is three-dimensional,\ with basis $\{\sigma^{i}\}_{i=1,2,3}$ and Lie brackets
$$\{\sigma^{i},\sigma^{j}\}=2\varepsilon_{ijk}\sigma^{k}.$$
In a local coordinate,\ we can set $\phi=\sum_{i=1}^{3}\phi_{i}\sigma^{i}$,\ where $\phi_{i}\in\Om^{1}(X)$.\ Then
$$0=\phi\wedge\phi=2(\phi_{1}\wedge\phi_{2})\sigma^{3}+2(\phi_{3}\wedge\phi_{1})\sigma^{2}+2(\phi_{2}\wedge\phi_{3})\sigma^{1}.$$
We have
\begin{equation}\label{90}
0=\phi_{1}\wedge\phi_{2}=\phi_{3}\wedge\phi_{1}=\phi_{2}\wedge\phi_{3}.
\end{equation}
On $Z^{c}$,\ $\phi$ is non-zero,\ then  without loss of generality we can assume that $\phi_{1}$ is non-zero.\ From (\ref{90}),\ there exist functions $\mu$ and $\nu$ such that
$$\phi_{2}=\mu\phi_{1}\ and\ \phi_{3}=\nu\phi_{1}.$$
Hence,
\begin{equation}\nonumber
\begin{split}
\phi&=\phi_{1}(\sigma^{1}+\mu\sigma^{2}+\nu\sigma^{3})\\
&=\phi_{1}(1+\mu^{2}+\nu^{2})^{1/2}(\frac{\sigma^{1}+\mu\sigma^{2}+\nu\sigma^{3}}{\sqrt{1+\mu^{2}+\nu^{2}}}).\\
\end{split}
\end{equation}
Then on $Z^{c}$ write $\phi=\xi\otimes\w$ for $\xi\in\Om^{0}(Z^{c},\mathfrak{g}_{P})$ with $\langle\xi,\xi\rangle=1$,\ and $\w\in\Om^{1}(Z^{c})$.\ We compute
$$0=d_{A}(\xi\otimes\w)=d_{A}\xi\wedge\w-\xi\otimes d\w,$$
$$0=d_{A}\ast(\xi\otimes\w)=d_{A}\xi\wedge\ast\w-\xi\otimes d\ast\w.$$
Taking the inner product with $\xi$ and using the consequence of $\langle\xi,\xi\rangle=1$ that $\langle\xi,d_{A}\xi\rangle=0$,\ we get $d\w=d^{\ast}{\color{red}\w}=0$.\ It follows that $d_{A}\xi\wedge\w=0$.\ Since $\w$ is nowhere zero along $Z^{c}$,\ we must have $d_{A}\xi=0$ along $Z^{c}$.\ Therefore,\ $A$ is reducible along $Z^{c}$.\ However according to \cite{DK} Lemma 4.3.21,\ $A$ is irreducible along $Z^{c}$.\ This is a contradiction unless $Z^{c}$ is empty.\ Therefore $Z=X$,\ so $\phi$ is identically zero.
\end{proof}
\section{Analytic results}

\subsection{The Kuranishi complex}

The most fundamental tool for understanding moduli space of anti-self-dual connection is the complex associate to an anti-self-dual connection $A_{asd}$ given by
$$0\ra\Om^{0}(\mathfrak{g}_{P})\xrightarrow{d_{A_{asd}}}\Om^{1}(\mathfrak{g}_{P})\xrightarrow{d^{+}_{A_{asd}}}{\Om^{2,+}(\mathfrak{g}_{P})}\ra 0.$$
The complex associated to Kapustin-Witten equations is the form
$$0\ra\Om^{1}(\mathfrak{g}_{P})\xrightarrow{d^{0}_{(A,\phi)}}\Om^{1}(\mathfrak{g}_{P})\times\Om^{1}(\mathfrak{g}_{P})
\xrightarrow{d^{1}_{(A,\phi)}}\Om^{2,-}(\mathfrak{g}_{P})\times\Om^{2,+}(\mathfrak{g}_{P})\ra 0,$$
where $d^{1}_{(A,\phi)}$ is the linearization of $KW$ at the configuration $(A,\phi)$,\ and $d^{0}_{(A,\phi)}$ gives the action of infinitesimal gauge transformations.\ These maps $d^{0}_{(A,\phi)}$ and $d^{1}_{(A,\phi)}$ form a complex whenever $KW(A,\phi)=0$.

The action of $g\in\G_{P}$ on $(A,\phi)\in \mathcal{A}_{P}\times\Om^{1}(\mathfrak{g}_{P})$,\ and the corresponding infinitesimal action of $\xi\in\Om^{0}(\mathfrak{g}_{P})$ is
$$d^{0}_{(A,\phi)}:\Om^{0}(\mathfrak{g}_{P})\ra\Om^{1}(\mathfrak{g}_{P})\times\Om^{1}(\mathfrak{g}_{P}),$$
\begin{equation}\nonumber
d^{0}_{(A,\phi)}(\xi)=(-d_{A}\xi,[\xi,\phi]).
\end{equation}
The linearization of $KW$ at the point $(A,\phi)$ is given by
\begin{equation}\nonumber
d^{1}_{(A,\phi)}(a,b)=\big{(}(d_{A}b+[a,\phi])^{-},(d_{A}a+[b,\phi])^{+}\big{)}.
\end{equation}
We compute
\begin{equation}\nonumber
d^{1}_{(A,\phi)}\circ d^{0}_{(A,\phi)}(\xi)=\big{(}[\xi,(d_{A}\phi)^{-}],[\xi,(F_{A}+\phi\wedge\phi)^{+}]\big{)}.
\end{equation}
The dual complex is
$$0\ra\Om^{2,-}(\mathfrak{g}_{P})\times\Om^{2,+}(\mathfrak{g}_{P})\xrightarrow{d^{1,\ast}_{(A,\phi)}}\Om^{1}(\mathfrak{g}_{P})
\times\Om^{1}(\mathfrak{g}_{P})\xrightarrow{d^{0,\ast}_{(A,\phi)}}\Om^{0}(\mathfrak{g}_{P})\ra 0.$$
There codifferentials are
\begin{equation}\nonumber
d^{1,\ast}_{(A,\phi)}(a',b')=(d_{A}^{\ast}b'-\ast[\phi,a'],d^{\ast}_{A}a'+\ast[\phi,b']),
\end{equation}
and
\begin{equation}\nonumber
d^{0,\ast}_{(A,\phi)}(a,b)=(-d^{\ast}_{A}a+\ast[\phi,\ast b]).
\end{equation}
\begin{prop}\label{26}
The map $KW(A,\phi)$ has an exact quadratic expansion given by
$$KW(A+a,\phi+b)=KW(A,\phi)+d^{1}_{A,\phi}(a,b)+\{(a,b),(a,b)\},$$
where $\{(a,b),(a,b)\}$ is the symmetric quadratic form given by
$$\{(a,b),(a,b)\}:=\big{(}[a,b]^{-},(a\wedge a+[b,b])^{+}\big{)}.$$
\end{prop}
Given fixed $(A_{0},\phi_{0})$,\ we look for solutions to the inhomogeneous equation $KW(A_{0}+a,\phi_{0}+b)=\psi_{0}$.\ By Proposition \ref{26},\ this equation is equivalent to
\begin{equation}
d^{1}_{A_{0},\phi_{0}}+\{(a,b),(a,b)\}=\psi_{0}-KW(A_{0},\phi_{0}).
\end{equation}
To make this equation elliptic,\ it's nature to impose the gauge-fixing condition
$$d^{0,\ast}_{(A_{0},\phi_{0})}(a,b)=\zeta.$$
If we define
\begin{equation}\nonumber
\begin{split}
&\D_{(A_{0},\phi_{0})}:=d^{0,\ast}_{(A_{0},\phi_{0})}+d^{1}_{(A_{0},\phi_{0})},\\
&\psi=\psi_{0}-KW(A_{0},\phi_{0}).\\
\end{split}
\end{equation}
then the elliptic system can be rewritten as
\begin{equation}\label{27}
\D_{(A_{0},\phi_{0})}+\{(a,b),(a,b)\}=(\zeta,\psi).
\end{equation}
This is situation is consider in \cite{FL2} equation 3.2 in the context of $PU(2)$ monopoles.

\subsection{Regularity and elliptic estimates}

First we summarize the result of \cite{FL2},\ which apply verbatim to Kapustin-Witten equations upon replacing the $PU(2)$ spinor $\Phi$ by $\phi$,
\begin{prop}(\cite{FL2} Corollary 3.4)
Let $X$ be a closed,\ oriented,\ $4$-dimensional Riemannian manifold,\ let $P\rightarrow X$ be a principal bundle with compact structure group,\ and let $(A_{0},\phi_{0})$ be a $C^{\infty}$ configuration in $\mathcal{C}_{P}$.\ Then there exist a positive constant $\epsilon=\epsilon(A_{0},\phi_{0})$ such that if $(a,b)$ is an $L^{2}_{1}$ solution to (\ref{27}) over $X$,\ where $(\zeta,\psi)$ is in $L^{2}_{k}$ and $\|(a,b)\|_{L^{4}}<\epsilon$,\ and $k\geq0$ is an integer,\ then $(a,b)\in L^{2}_{k+1}$ and there is a polynomial $Q_{k}(x,y)$,\ with positive real coefficients,\ depending at most on $(A_{0},\phi_{0}),k$
such that $Q_{k}(0,0)=0$ and
\begin{equation}
\|(a,b)\|_{L^{2}_{k+1,A_{0}}(X)}\leq Q_{k}\big{(}\|(\zeta,\psi)\|_{L^{2}_{k,A_{0}}(X)},\|(a,b)\|_{L^{2}(X)}\big{)}.
\end{equation}
In particular,\ if $(\zeta,\psi)$ is in $C^{\infty}$ and if $(\zeta,\psi)=0$,\ then
$$\|(a,b)\|_{L^{2}_{k+1,A_{0}}(X)}\leq C\|(a,b)\|_{L^{2}(X)}.$$
\end{prop}
\begin{prop}\label{31}
Let $X$ be a closed,\ oriented,\ $4$-dimensional Riemannian manifold,\ let $P\rightarrow X$ be a principal bundle with compact structure group.\ Suppose $\Om\subset X$ is an open subset such that $P|_{\Om}$ is trivial,\ and $\Ga$ is a smooth flat connection.\ Then there exist a positive constant $\epsilon=\epsilon(\Om)$ with the following significance.\ Suppose that $(a,b)$ is an $L^{2}_{1}(\Om)$ solution to the elliptic system (\ref{27}) over $\Om$ with $(A_{0},B_{0})=(\Ga,0)$,\ where $(\zeta,\psi)$ is in $L^{2}_{k}(\Om)$,\ $k\geq1$ is an integer,\ and $\|(a,b)\|_{L^{4}}<\epsilon$.\ Let $\Om'\Subset\Om$ be a precompact open subset.\ Then $(a,b)$ is in $L^{2}_{k+1}(\Om')$ and there is a polynomial $Q_{k}(x,y)$,\ with positive real coefficients,\ depending at most on $k,\Om,\Om'$ such that $Q_{k}(0,0)=0$ and
\begin{equation}
\|(a,b)\|_{L^{2}_{k+1,\Ga}(\Om')}\leq Q_{k}\big{(}\|(\zeta,\psi)\|_{L^{2}_{k,A_{0}}(X)},\|(a,b)\|_{L^{2}(X)}\big{)}.
\end{equation}
If $(\zeta,\psi)$ is in $C^{\infty}(\Om)$  then $(\zeta,\psi)$ is in $C^{\infty}(\Om')$ and if $(\zeta,\psi)=0$,\ then
$$\|(a,b)\|_{L^{2}_{k+1,\Ga}(\Om')}\leq C\|(a,b)\|_{L^{2}(\Om)}.$$
\end{prop}
We assume that $\int_{B_{r}(x)}|F_{A}|^{2}\leq\epsilon$ any $x\in X$ and a $0<r\leq\delta$.\ We then use the following version of the Uhlenbeck theorem (the original appeared in \cite{U}) stated in Remark 6.2a of \cite{KW} and proved in Pages 105-106 of \cite{KW} by Wehrheim.
\begin{thm}\label{23}
Let $X$ be a compact four-dimensional manifold,\ let $P\rightarrow X$ be a principal bundle with compact structure group,\ and let $2\leq p<4$.\ Let $B_{r}(x)$ denote the geodesic ball of radius $r$ centered at $x$.\ Then there exists constant $C,\epsilon>0$ such that the following holds:

For every point $x\in X$,\ there exists a positive radius $r_{x}$ such for all $r\in(0,r_{x}]$,\ all smooth flat connections $\Ga\in\Om^{1}({B_{r}(x)},\mathfrak{g}_{P})$,\ and all $L^{p}_{1}$ connection $A\in\Om^{1}({B_{r}(x)},\mathfrak{g}_{P})$ with $\|F_{A}\|_{L^{p}(B_{r}(x))}\leq\epsilon$,\ there exists a gauge transformation $g\in\G^{2,p}(B_{r}(x))$ such that
\begin{equation}\nonumber
\begin{split}
&(1)\ d^{\ast}_{\Ga}\big{(}g^{\ast}(A)-\Ga\big{)}=0\ on\ B_{r}(x)\ and\ \frac{\pa}{\pa r}\lrcorner(g^{\ast}A-\Ga)=0\ on\ \pa B_{r}(x),\ and\\
&(2)\|g^{\ast}A-\Ga\|_{L^{2}_{1}(B_{r}(x))}\leq C\|F_{A}\|_{L^{2}(B_{r}(x))}.\\
\end{split}
\end{equation}
\end{thm}
At this point,\ we must deviate slightly from \cite{FL2},\ since we have no estimate of the form $|\phi|^{4}\leq C|\phi\wedge\phi|^{2}$ (c.f. \cite{FL2} Lemma 2.26),\ so $F^{+}_{A}$ does not bound $\phi$.\ Instead,\ we get the following analogue of \cite{FL2} Corollary 3.15 by combining Proposition \ref{31} and Theorem \ref{23}.
\begin{prop}
Let $B$ be the open unit ball with center at the origin,\ let $U\Subset B$ be an open subset,\ let $P\rightarrow B$ be a principal bundle with compact structure group,\ and let $\Ga$ be a smooth flat connection on $P$.\ Then there is a positive constant $\epsilon$ such that for all integers $k\geq1$ there is a constant $C(k,U)$ such that for all $L^{2}_{1}$ solution $(A,\phi)$ satisfying
$$\|F_{A}\|^{2}_{L^{2}(B)}+\|\phi\|^{4}_{L^{4}(B)}\leq\epsilon,$$
there is an $L^{2}_{2}$ gauge transformation $g$ such that $g^{\ast}(A,\phi)$ is in $C^{\infty}(B)$ with
$$d^{\ast}_{\Ga}\big{(}g^{\ast}(A)-\Ga\big{)}=0$$
over $B$ and
$$\|g^{\ast}(A,\phi)\|_{L^{2}_{k,\Ga}(U)}\leq C(\|F_{A}\|_{L^{2}(B)}+\|\phi\|_{L^{2}(B)}).$$
\end{prop}
\begin{proof}
By choosing $\epsilon$ as in Theorem \ref{23},\ we can find $g$ such that
$$d^{\ast}_{\Ga}\big{(}g^{\ast}(A)-\Ga\big{)}=0$$
and
$$\|g^{\ast}(A)-\Ga\|_{L^{2}_{1}(B)}\leq C\|F_{A}\|_{L^{2}(B)}.$$
By the Sobolev embedding theorem,
$$\|g^{\ast}A-\Ga\|_{L^{4}(B)}\leq C\|F_{A}\|_{L^{2}(B)}.$$
Upon taking $(a,b)=(g^{\ast}A-\Ga,\phi)$,\ we are in the situation of Proposition \ref{31}.\ Thus we get the desired estimate.
\end{proof}
We generalize this estimate for geodesic ball:
\begin{thm}\label{53}
Let $X$ be an oriented Riemannian four-manifold,\ and let $P\rightarrow X$ be a principal bundle with compact structure group.\ Let $B_{r}(x)$ denote the geodesic ball of radius $r$ centered at $x$,\ and fix any $\a\in(0,1]$.\ For all $k\geq 1$ there exists constants $C(\a,k,r),\epsilon$ such that the following holds:

For all point $x\in X$,\ there exists a positive radius $r_{x}$ such that for all $r\in(0,r_{x}]$,\ all smooth flat connection $\Ga\in\mathcal{A}_{P}(B_{r}(x))$,\ and all $L^{2}_{1}$ solution $(A,\phi)$ with
$$\|F_{A}\|^{2}_{L^{2}(B_{r}(x))}+\|\phi\|^{4}_{L^{4}(B_{r}(x))}\leq\epsilon,$$
there is an $L^{2}_{2}(B_{r}(x))$ gauge transformation $g$ such that $g^{\ast}(A,\phi)$ is in $C^{\infty}(B_{r}(x))$ with $$d^{\ast}_{\Ga}\big{(}g^{\ast}(A)-\Ga\big{)}=0$$
over $B_{r}(x)$ and
$$\|g^{\ast}(A,\phi)\|_{L^{2}_{k,\Ga}(B_{\a r}(x))}\leq C(\|F_{A}\|_{L^{2}(B_{r}(x))}+\|\phi\|_{L^{2}(B_{r}(x))}).$$
\end{thm}
We will show that all $L^{2}_{1}$ solutions to Kapustin-Witten equations are $L^{2}_{2}$-gauge equivalent to a smooth solution.\ The way which is similar to Mares \cite{BM} dealt with Vafa-Witten equations.
\begin{thm}\label{42}
Let $X$ be a closed smooth Riemannian four-manifold,\ $P\rightarrow X$ is a smooth principal $G$-bundle with $G$ compact and connected,\ $(A,\phi)$ is an $L^{2}_{1}$ configuration,\ and $KW(A,\phi)=0$.\ Then $(A,\phi)$ is $L^{2}_{2}$-gauge equivalent to a smooth configuration.
\end{thm}
\begin{proof}
By gauge-fixing on small ball $B_{r}(x)$ in which the local regularity theorem applies,\ we get $L^{2}_{2}$ trivializations $h_{1,x}$ of $P$ over $B_{r}(x)$ such that $h_{1,x}(A,\phi)$ is smooth.\ Since the transition functions $h_{1,x'}h^{-1}_{1,x}$ intertwine smooth connections,\ they defines a smooth principal $G$-bundle $P'$.\ The trivializations $h_{1,x}$ patch together to define an $L^{2}_{2}$ isomorphism $h_{1}:P\rightarrow P'$.\ The $h_{1,x}(A,\phi)$ determine a smooth configuration $(A',\phi')$ in $P'$ such that $h(A,\phi)=(A',\phi')$.

In order to prove that $(A,\phi)$ is $L^{2}_{2}$-gauge equivalent to a smooth connection,\ it suffices to show that there exists a smooth isomorphism $h_{2}:P\rightarrow P'$,\ for then $g:=h^{-1}_{2}h_{1}\in\G_{P}^{2,2}$ is the desired gauge transformation.\ The existence of $h_{2}$ is a consequence of \cite{BM} Theorem 3.3.10.
\end{proof}

\section{Gap phenomenon for extra fields}

\subsection{Uniform positive lower bounds for the least eigenvalue of $d^{+}_{A}d^{+,\ast}_{A}$}

\begin{defi}\label{111}(\cite{T4} Definition 3.1)
Let $X$ be a compact $4$-dimensional Riemannain manifold and $P\rightarrow X$ be a principal $G$-bundle with $G$ being a compact Lie group.\ Let $A$ be a connection of Sobolev class $L^{2}_{1}$ on $P$.\ The least eigenvalue of $d^{+}_{A}d^{+,\ast}_{A}$ on $L^{2}(X;\Om^{+}(\mathfrak{g}_{P}))$ is
\begin{equation}\label{12}
\mu(A):=\inf_{v\in\Om^{+}(\mathfrak{g}_{P})\backslash\{0\}}\frac{\|d^{+,\ast}_{A}v\|^{2}}{\|v\|^{2}}.
\end{equation}
\end{defi}
For a Riemannian metric $g$ on a 4-manifold,\ $X$,\ let $R_{g}(x)$ denotes its scalar curvature at a point $x\in X$ and let $\W^{\pm}_{g}\in End(\Om_{x}^{\pm})$ denote its self-dual and anti-self-dual Weyl curvature tensors at $x$,\ where $\Om^{2}_{x}=\Om^{+}_{x}\oplus\Om^{-}_{x}$.\ Define
$$\w_{g}^{\pm}:=Largest\ eigenvalue\ of\ \W^{\pm}_{g}(x),\ \forall x\in X.$$
We recall the following Weitenb\"{o}ck formula,
\begin{equation}\label{28}
2d^{+}_{A}d^{+,\ast}_{A}v=\na_{A}^{\ast}\na_{A}v+\big{(}\frac{1}{3}R_{g}-2\w_{g}^{+}\big{)}v+\{F^{+}_{A},v\},\ \forall v\in\Om^{+}(\mathfrak{g}_{P}).
\end{equation}
We called a Riemannian metric,\ $g$,\ on $X$ $positive$ if $\frac{1}{3}R_{g}-2\w^{+}_{g}>0$,\ that is,\ the operator $R_{g}/3-2\W^{+}_{g}\in End(\Om^{+})$ is pointwise positive definite.\ Then the Weitzenb\"{o}ck formula (\ref{28}) ensures that the least eigenvalue function,
\begin{equation}\label{44}
\mu[\cdot]:\ M_{ASD}(P,g)\rightarrow[0,\infty),
\end{equation}
defined by $\mu(A)$ in (\ref{12}),\ admits a uniform positive lower bound,\ $\mu_{0}$,
$$\mu(A)\geq\mu_{0},\ \forall[A]\in\ M_{ASD}(P,g).$$

The existence of a uniform positive lower bound for the least eigenvalue function (\ref{44}) is more subtle and relies on the generic metric theorem Freed and Uhlenbeck (\cite{FU},\ Pages 69-73),\ together with certain extensions due to Donaldson and Kronheimer (\cite{DK},\ Sections 4.3.3.).\ Under suitable hypotheses on $P$ and a generic Riemannian metric,\ $g$,\ on $X$,\ their results collectively ensure that $\mu(A)>0$ for all $[A]$ in both $M_{ASD}(P,g)$ and every $M_{ASD}(P_{l},g)$,\ appearing in its $Uhlenbeck\ compactification$ (see \cite{DK} Definition 4.4.1,\ Condition 4.4.2,\ and Theorem 4.4.3),
\begin{equation}\label{45}
\bar{M}_{ASD}(P,g)\subset\bigcup_{i=1}^{L}\big{(}M_{ASD}(P_{l},g)\times Sym^{l}(X)\big{)},
\end{equation}
where $L=L\big{(}k(P)\big{)}\geq0$ is a sufficiently large integer.
\begin{thm}\label{50}(\cite{PF1} Theorem 34.23,\ \cite{PF2} Theorem 3.5)
Let $G$ be a compact,\ simple Lie group and $P$ a principal $G$-bundle over a closed,\ connected,\ four-dimensional,\ oriented,\ smooth manifold,\ $X$.\ Then there is an open dense subset,\ $\mathscr{C}(X,p_{1}(P))$,\ of Banach space,\ $\mathscr{C}(X)$,\ of conformal equivalence classes,\ $[g]$,\ of $C^{r}$ Riemannian metrics on $X$ (for some $r\geq3)$ with the following significance.\ Assume that $[g]\in\mathscr{C}(X)$ and at least one of the following holds:

(1)$b^{+}(X)=0$,\ the group $\pi_{1}(X)$ has no non-trivial representations in $G$,\ and $G=SU(2)$ or $SO(3)$;\ or

(2)$b^{+}(X)>0$,\ the group $\pi_{1}(X)$ has no non-trivial representations in $G$,\ and $G=SO(3)$ and the second Stiefel-Whitney class,\ $w_{2}(P)\in H^{2}(X;\mathbb{Z}/2\mathbb{Z})$,\ is non-trivial.

Then every point $[A]\in M(P,g)$ has the property that $\mu(A)>0$.
\end{thm}
For a small enough $\varepsilon(g,k(P))\in(0,1]$,\ if $\|F^{+}_{A}\|_{L^{2}_{1}}\leq\varepsilon$,\ the eigenvalue $\mu(A)$ also has a non-zero bound.
\begin{cor}\label{43}(\cite{PF1} Corollary 34.28,\ \cite{PF2} Corollary 3.9)
Assume the hypotheses of Theorem \ref{50} and that $g$ is generic.\ Then these are constants,\ $\de=\de (P,g)\in(0,1]$ and $\mu_{0}=\mu_{0}(P,g)>0$,\ such that
\begin{equation}\nonumber
\begin{split}
&\mu(A)\geq\mu_{0},\ [A]\in M_{ASD}(P,g),\\
&\mu(A)\geq\frac{\mu_{0}}{2},\ [A]\in\mathfrak{B}_{\de}(P,g),\\
\end{split}
\end{equation}
where $\mathfrak{B}_{\de}(P,g):=\{[A]:\|F^{+}_{A}\|_{L^{2}(X)}<\de\}.$
\end{cor}
\subsection{Uniform positive lower bounds for extra fields}

Let $A$ be fixed,\ any connection $B$ can be written uniquely as
$$B=A+a\ with\ a\in\Om^{1}(\mathfrak{g}_{P}).$$
Therefore if $B$ has anti-self-dual curvature,\ then
\begin{equation}\label{13}
0=F^{+}_{A}+d^{+}_{A}a+(a\wedge a)^{+}.
\end{equation}
Conversely,\ if $a\in\Ga(\Om^{1}(\mathfrak{g}_{P}))$ satisfies (\ref{13}),\ then $B=A+a$ has anti-self-dual curvature.\ Because the operator $d^{+}_{A}$ is not properly elliptic,\ it is convenient to write $a=d^{+,\ast}_{A}u$ for $u\in\Om^{+}(\mathfrak{g}_{P})$ and replace (\ref{13}) by
\begin{equation}\label{14}
d^{+}_{A}d^{+,\ast}_{A}u+(d^{+,\ast}_{A}u\wedge d^{+,\ast}_{A}u)^{+}=-F^{+}_{A}.
\end{equation}
(\ref{14}) is properly elliptic system.\ Notice that if $A$ is anti-self-dual,\ then (\ref{14}) automatically has a solution,\ namely $u=0$.\ If $F^{+}_{A}$ is small in an appropriate norm,\ but non-zero,\ it is still reasonable to assume that (\ref{14}) has a solution $u$ which also small.

If $\G(\cdot,\cdot)$ denotes the Green kernel of the Laplace operator,\ $d^{\ast}d$,\ on $\Om^{2}(X)$,\ we define
\begin{equation}\nonumber
\begin{split}
&\|v\|_{L^{\sharp}(X)}:=\sup_{x\in X}\int_{X}\G(x,y)|v|(y)dvol_{g}(y),\\
&\|v\|_{L^{\sharp,2}(X)}:=\|v\|_{L^{\sharp}(X)}+\|v\|_{L^{2}(X)},\ \forall v\in\Om^{2}(\mathfrak{g}_{P}).\\
\end{split}
\end{equation}
We recall that $\G(x,y)$ has a singularity comparable with $dist_{g}(x,y)^{-2}$,\ when $x,y\in X$ are close \cite{IC}.\ The norm $\|v\|_{L^{2}(X)}$ is conformally invariant and $\|v\|_{L^{\sharp}(X)}$ is scale invariant.\ One can show that $\|v\|_{L^{\sharp}(X)}\leq c_{p}\|v\|_{L^{p}(X)}$ for every $p>2$,\ where $c_{p}$ depends at most on p and the Riemanian metric,\ $g$,\ on $X$.
\begin{thm}(\cite{FL1} Proposition 7.6)\label{15}
Let $G$ be a compact Lie group,\ $P$ a principal $G$-bundle over a compact,\ connected,\ four-dimensional manifold,\ $X$,\ with Riemannian metric,\ $g$,\ and $E_{0},\mu_{0}\in(0,\infty)$ constants.\ Then there are constants,\ $C_{0}=C_{0}(E_{0},g,\mu_{0})\in(0,\infty)$ and $\eta=\eta(E_{0},g,\mu_{0})\in(0,1]$,\ with the following significance.\ If $A$ is a $C^{\infty}$ connection on $P$ such that
\begin{equation}\nonumber
\begin{split}
&\mu(A)\geq\mu_{0},\\
&\|F_{A}^{+}\|_{L^{\sharp,2}(X)}\leq\eta,\\
&\|F_{A}\|_{L^{2}(X)}\leq E_{0},\\
\end{split}
\end{equation}
then there is a anti-self-dual connection,\ $A_{asd}$ on $P$, of class $C^{\infty}$ such that
$$\|A_{asd}-A\|_{L^{2}_{1}(X)}\leq C_{0}\|F^{+}_{A}\|_{L^{\sharp,2}(X)}.$$
\end{thm}
Theorem \ref{15} requires that $F^{+}_{A}$ is small in the sense that $\|F_{A}\|_{L^{\sharp,2}(X)}\leq\varepsilon$,\ for a suitable $\varepsilon\in(0,1]$.
By the definition of $\|F_{A}^{+}\|_{L^{\sharp,2}(X)}$,\ we have $$\|F_{A}^{+}\|_{L^{\sharp,2}(X)}\leq C(p,X)\|F_{A}^{+}\|_{L^{\infty}}.$$
Since $\|F_{A}\|^{2}_{L^{2}}=2\|F_{A}^{+}\|^{2}_{L^{2}}+8\pi^{2}k(P)$,\ where
$$k(P):=-\frac{1}{8\pi^{2}}\int_{X}tr(F_{A}\wedge F_{A})\in\mathbb{Z}.$$
then we can choose $\|F_{A}^{+}\|_{L^{\infty}}$ sufficiently small such that $\|F_{A}^{+}\|_{L^{\sharp,2}(X)}$ and $\|F_{A}\|_{L^{2}}$ satisfy the conditions in theorem \ref{15}.
\begin{cor}\label{17}
Assume the hypotheses of Theorem \ref{15}.\ Then there are constants,\ $C=C(g,\mu_{0})\in(0,\infty)$ and $\varepsilon=\varepsilon(g,\mu_{0})\in(0,1]$,\ with the following significance.\ If $A$ is a $C^{\infty}$ connection on $P$ such that
\begin{equation}\nonumber
\begin{split}
&\mu(A)\geq\mu_{0},\\
&\|F^{+}_{A}\|_{L^{\infty}(X)}\leq\varepsilon,\\
\end{split}
\end{equation}
then there is a anti-self-dual connection,\ $A_{asd}$ on $P$, of class $C^{\infty}$ such that
$$\|A_{asd}-A\|_{L^{2}_{1}(X)}\leq C\|F_{A}^{+}\|_{L^{\infty}(X)}.$$
\end{cor}
\begin{lem}(\cite{PF1} Lemma 34.6,\ \cite{PF2} Lemma A.2)\label{58}
Let $X$ be a closed,\ four-dimensional,\ oriented,\ smooth manifold with Riemannian metric,\ $g$,\ and $q\in[4,\infty)$.\ Then there are positive constants,\ $c=c(g,q)\geq1$ and $\varepsilon=\varepsilon(g,q)\in(0,1]$,\ with the following significance.\ Let $r\in[4/3,2)$ be defined by $1/r=2/d+1/q$.\ Let $G$ be a compact Lie group and $A$ a connection of class $C^{\infty}$ on a principal bundle $P$ over $X$ that obeys the curvature $$\|F_{A}\|_{L^{2}(X)}\leq\varepsilon.$$
If $v\in\Om^{2,+}(X,\mathfrak{g}_{P})$,\ then
$$\|v\|_{L^{q}(X)}\leq c(\|d^{+}_{A}d^{+,\ast}_{A}v\|_{L^{r}(X)}+\|v\|_{L^{r}(X)}).$$
\end{lem}
\begin{lem}\label{59}
Let $G$ be a compact Lie group,\ let $P$ be a principal $G$-bundle over a closed,\ four-dimensional,\ oriented,\ smooth manifold,\ $X$,\ with Riemannian metric,\ $g$,\ and $\mu_{0}\in(0,\infty)$ a constant,\ let $p\in[2,4)$ and $q\in[4,\infty)$ is defined by $1/p=1/4+1/q$.\ Then there exists constants,\ $\de=\de(g,p)\in(0,1)]$ and $C=C(g,p,\mu)\in[1,\infty)$,\ with the following significance.\ If $A=A_{asd}-d^{+,\ast}_{A}u$ is a connection on $P$ such that
\begin{equation}\nonumber
\begin{split}
&\mu(A)\geq\mu_{0},\\
&\|d_{A}^{+,\ast}u\|_{L^{2}_{1}(X)}\leq\de,\\
\end{split}
\end{equation}
where $A_{asd}$ is a anti-self-dual connection.\ Then
\begin{equation}\label{114}
\|d^{+,\ast}_{A}u\|_{L^{q}(X)}\leq C\|F^{+}_{A}\|_{L^{p}(X)}.
\end{equation}
\end{lem}
\begin{proof}
By the anti-self-dual equation,\ $F^{+}(A+d^{+,\ast}_{A}u)=0$,\ we obtain
\begin{equation}\label{48}
d_{A}^{+}d^{+,\ast}_{A}u+(d^{+,\ast}_{A}u\wedge d^{+,\ast}_{A}u)^{+}=-F_{A}^{+}.
\end{equation}
Then we have
\begin{equation}\nonumber
\begin{split}
\|F^{+}_{A}\|_{L^{2}(X)}&\leq2\|d^{+,\ast}_{A}u\|^{2}_{L^{4}(X)}+\|d_{A}^{+}d^{+,\ast}_{A}u\|_{L^{2}(X)}\\
&\leq 2\|d^{+,\ast}_{A}u\|^{2}_{L^{2}_{1}(X)}+\|d^{+,\ast}_{A}u\|^{2}_{L^{2}_{1}(X)}.\\
\end{split}
\end{equation}
We can choose constant $\de$ small enough,\ such that
$$\|F^{+}_{A}\|_{L^{2}(X)}\leq\varepsilon(g),$$
where the constant $\varepsilon(g)$ is the constant in Lemma \ref{58}.\ Then we have a priori estimate for all $v\in\Om^{2,+}(X,\mathfrak{g}_{P})$,
$$\|v\|_{L^{2}_{1}(X)}\leq c\|d^{+}_{A}d^{+,\ast}_{A}v\|_{L^{4/3}(X)}+\|v\|_{L^{2}(X)}.$$
By Sobolev imbedding $L^{2}_{1}\hookrightarrow L^{p}$ ($p\leq4)$,\ we have
\begin{equation}\label{49}
\|v\|_{L^{p}(X)}\leq c_{p}\|d^{+}_{A}d^{+,\ast}_{A}v\|_{L^{4/3}(X)}+\|v\|_{L^{2}(X)}.
\end{equation}
For $p\in[2,4)$ and $q\in[4,\infty)$,\ equation (\ref{48}) gives
\begin{equation}\nonumber
\begin{split}
\|d^{+}_{A}d^{+,\ast}_{A}u\|_{L^{p}(X)}&\leq\|F_{A}^{+}\|_{L^{p}(X)}+2\|d^{+,\ast}_{A}u\|_{L^{4}(X)}\|d^{+,\ast}_{A}u\|_{L^{q}(X)}\\
&\leq\|F_{A}^{+}\|_{L^{p}(X)}+2\de\|d^{+,\ast}_{A}u\|_{L^{q}(X)}\\
&\leq\|F_{A}^{+}\|_{L^{p}(X)}+2\de\|\na_{A}u\|_{L^{q}(X)}\\
&\leq\|F_{A}^{+}\|_{L^{p}(X)}+2\de C_{s}\|\na_{A}u\|_{L^{p}_{1,A}(X)}\\
&\leq\|F_{A}^{+}\|_{L^{p}(X)}+2\de C_{s}\|u\|_{L^{p}_{2,A}(X)}.\\
\end{split}
\end{equation}
where we have applied the Sobolev embedding $L^{p}_{1}\hookrightarrow L^{q}$ with embedding constant $C_{s}$ and Kato Inequality.

We have a priori $L^{p}$ estimate for the elliptic operator,\ $d^{+}_{A}d^{+,\ast}_{A}$,\ namely
$$\|u\|_{L^{p}_{2,A}(X)}\leq C_{1}(\|d^{+}_{A}d^{+,\ast}_{A}u\|_{L^{p}(X)}+\|u\|_{L^{p}(X)}).$$
Since $p\in[2,4)$,\ $\|u\|_{L^{4/3}(X)}\leq c\|u\|_{L^{2}(X)}\leq c\mu(A)^{-1}\|d^{+}_{A}d^{+,\ast}_{A}u\|_{L^{p}(X)},$\ then by (\ref{49}),\ we obtain
\begin{equation}\nonumber
\begin{split}
\|u\|_{L^{p}(X)}&\leq c_{p}(\|d^{+}_{A}d^{+,\ast}_{A}u\|_{L^{4/3}(X)}+\|u\|_{L^{2}(X)})\\
&\leq c_{p}(\|d^{+}_{A}d^{+,\ast}_{A}u\|_{L^{p}(X)}+\|v\|_{L^{2}(X)})\\
&\leq c_{p}(\|d^{+}_{A}d^{+,\ast}_{A}u\|_{L^{p}(X)}+\mu(A)^{-1}\|d^{+}_{A}d^{+,\ast}_{A}u\|_{L^{2}(X)})\\
&\leq c_{p}(1+\mu(A)^{-1})\|d^{+}_{A}d^{+,\ast}_{A}u\|_{L^{p}(X)}.\\
\end{split}
\end{equation}
Combing the preceding inequalities gives
\begin{equation}\nonumber
\begin{split}
\|u\|_{L^{p}_{2,A}(X)}&\leq C_{1}(\|d^{+}_{A}d^{+,\ast}_{A}u\|_{L^{p}(X)}+c_{p}(1+\mu(A)^{-1})\|d^{+}_{A}d^{+,\ast}_{A}u\|_{L^{2}(X)})\\
&\leq C_{2}\|d^{+}_{A}d^{+,\ast}_{A}u\|_{L^{p}(X)}\\
&\leq C_{2}(\|F^{+}_{A}\|_{L^{p}(X)}+\de C_{s}\|u\|_{L^{p}_{2,A}(X)}).\\
\end{split}
\end{equation}
Thus,\ for small enough $\de$ such that $C_{2}\de C_{s}<\frac{1}{2}$,\ rearrangement yields
\begin{equation}
\|u\|_{L^{p}_{2,A}(X)}\leq 2C_{2}\|F^{+}_{A}\|_{L^{p}(X)}.
\end{equation}
The inequality (\ref{114}) follows from
$$\|d^{+,\ast}_{A}u\|_{L^{q}(X)}\leq \kappa_{p}\|d^{+,\ast}_{A}u\|_{L^{p}_{1,A}(X)}\leq\kappa_{p}\|u\|_{L^{p}_{2,A}(X)}.$$
\end{proof}

We now have all the ingredients required to conclude the

\textbf{Proof of Theorem \ref{2}.}
By Theorem \ref{42},\ we can assume $(A,\phi)$ is smooth.\ From the Theorem \ref{16},\ we have
$$\|\phi\|_{L^{\infty}(X)}\leq C_{1}\|\phi\|_{L^{2}(X)},$$
where $C_{1}$ is only dependent on manifold.\ Since $(A,\phi)$ is a solution of Kapustin-Witten equations,\ then we have
$$\|F_{A}^{+}\|_{L^{\infty}(X)}=\|\phi\wedge\phi\|_{L^{\infty}(X)}\leq C_{2}\|\phi\|^{2}_{L^{2}(X)},$$
where $C_{2}$ is only dependent on manifold.
We can choose $\|\phi\|_{L^{2}(X)}\leq C$ sufficiently small such that $C_{2}C^{2}\leq\varepsilon$,\ where $\varepsilon$ is the constant in the hypothesis of
Corollary \ref{17},\ then
$$\|F_{A}^{+}\|_{L^{\infty}(X)}\leq\varepsilon.$$
From Corollary \ref{17},\ there exist a anti-self-dual connection $A_{0}$ such that
$$\|A-A_{0}\|_{L^{2}_{1}(X)}\leq C_{3}\|F^{+}_{A}\|_{L^{\infty}(X)}.$$
We can choose $\|\phi\|_{L^{2}(X)}\leq C$ sufficiently small such that $C_{2}C^{2}\leq\de$,\ where $\de$ is the constant in the hypothesis of Lemma \ref{59}.\ By Lemma \ref{59} then we obtain
$$\|A-A_{0}\|_{L^{q}(X)}\leq C_{3}\|F^{+}_{A}\|_{L^{p}(X)}.$$
where $1/p=1/4+1/q$,\ $p\in[2,4)$ and $q\in[4,\infty)$.\ Choosing $q=4$,\ $p=2$,\ hence we have
\begin{equation}
\|A-A_{0}\|_{L^{4}(X)}\leq C_{3}\|F^{+}_{A}\|_{L^{2}(X)}.
\end{equation}
Using the Weitezenb\"{o}ck formula \cite{FU} (6.25),\ we have
\begin{equation}\label{18}
(2d^{-,\ast}_{A_{0}}d^{-}_{A_{0}}+d_{A_{0}}d_{A_{0}}^{\ast})\phi=\na_{A_{0}}^{\ast}\na_{A_{0}}\phi+Ric\circ\phi.
\end{equation}
which provides  an integral inequality
\begin{equation}\label{19}
\|\na_{A_{0}}\phi\|^{2}_{L^{2}(X)}+\int_{X}\langle Ric\circ\phi,\phi\rangle\geq0.
\end{equation}
By the Weitezenb\"{o}ck formula \cite{FU} (6.25) again,\ we have
\begin{equation}\label{20}
(2d^{-,\ast}_{A}d^{-}_{A}+d_{A}d_{A}^{\ast})\phi=\na_{A}^{\ast}\na_{A}\phi+Ric\circ\phi+\ast[F^{+}_{A},\phi].
\end{equation}
Since $(d_{A}\phi)^{-}=d^{\ast}_{A}\phi=0$,\ we also obtain an integral equality
\begin{equation}\label{21}
\|\na_{A}\phi\|^{2}_{L^{2}(X)}+\int_{X}\langle Ric\circ\phi,\phi\rangle+4\|F^{+}_{A}\|^{2}=0.
\end{equation}
We have another integral inequality
\begin{equation}\label{22}
\begin{split}
\|\na_{A}\phi-\na_{A_{0}}\phi\|^{2}_{L^{2}(X)}&\leq\|[A-A_{0},\phi]\|^{2}_{L^{2}(X)}\\
&\leq C_{4}\|A-A_{0}\|^{2}_{L^{4}(X)}\|\phi\|^{2}_{L^{4}(X)}\\
&\leq C_{5}\|F^{+}_{A}\|^{2}_{L^{2}(X)}\|\phi\|^{2}_{L^{2}(X)}.\\
\end{split}
\end{equation}
Combing the preceding inequalities gives
\begin{equation}\nonumber
\begin{split}
0&\leq\|\na_{A_{0}}\phi\|^{2}_{L^{2}(X)}+\int_{X}\langle Ric\circ\phi,\phi\rangle\\
&\leq\|\na_{A}\phi\|^{2}_{L^{2}(X)}+\int_{X}\langle Ric\circ\phi,\phi\rangle+\|\na_{A}\phi-\na_{A_{0}}\phi\|_{L^{2}(X)}^{2}\\
&\leq(C_{6}\|\phi\|^{2}_{L^{2}(X)}-4)\|F^{+}_{A}\|^{2}_{L^{2}(X)}.\\
\end{split}
\end{equation}
We choose $\|\phi\|_{L^{2}(X)}\leq C$ sufficiently small such that $C_{6}\|\phi\|^{2}\leq2$,\ then
\begin{equation}\nonumber
F^{+}_{A}\equiv0.
\end{equation}

\textbf{Proof of Corollary \ref{92}.} The conclusion follow from Theorem \ref{2} and positive uniform lower bound on $\mu(A)$ provided by Corollary \ref{43}.

\section{Vafa-Witten equations}

In search of evidence for S-duality,\ Vafa and Witten explored their twist of $\mathcal{N}=4$ supersymmetric Yang-Mills theory \cite{VW}.\ Vafa-Witten introduced a set of gauge-theoretic equations on a $4$-manifold,\ the moduli space of solutions to the equations is expected to produce a possibly new invariant of some kind.\ The equations we consider involve a triple consisting of a connection and other extra fields coming from a principle bundle over $4$-manifold.\ We consider the following equations for a triple $(A,B,C)\in\mathcal{A}\times\Om^{2,+}(X,\mathfrak{g}_{P})\times\Om^{0}(X,\mathfrak{g}_{P})$,\
$$d_{A}C+d^{\ast}_{A}B=0,$$
$$F^{+}_{A}+\frac{1}{8}[B.B]+\frac{1}{2}[B,C]=0.$$
where $[B.B]\in\Om^{2,+}(X,\mathfrak{g}_{P})$ is defined in \cite{BM} Appendix A.\ We also define the gauge-equivariant map
$$VW(A,B,C)=\big{(}d_{A}C+d^{\ast}_{A}B,F^{+}_{A}+\frac{1}{8}[B.B]+\frac{1}{2}[B,C]\big{)}.$$
Mimicking the setup of Donaldson theory,\ the $VW$-moduli space is
$$M_{VW}(P,g):=\{(A,B,C)\mid VW(A,B,C)=0\}/\mathcal{G}_{P}.$$
We are interesting in the case $C=0$,\ for which the equations reduce to
$$d^{\ast}_{A}B=0,$$
$$F^{+}_{A}+\frac{1}{8}[B.B]=0.$$
\begin{thm}(\cite{BM} Theorem 2.1.1)
Let $X$ be a closed,\ oriented,\ four-manifold,\ $(A,B,C)$ be a solution of the Vafa-Witten equation.\ If $A$ be an irreducible $SU(2)$ or $SO(3)$ connection,\ then
$C=0$.
\end{thm}
In \cite{BM},\ Mares obtained a bound on $\|B\|_{L^{\infty}(X)}$ in terms of $\|B\|_{L^{2}(X)}$.
\begin{thm}(\cite{BM} Theorem 3.1.1)
Let $X$ be a closed,\ oriented,\ smooth,\ four-dimensional manifold.\ There exists a constant $\la=\la(X)$ with the following property.\ For any principal bundle $P\rightarrow X$ and any $L^{2}_{1}$-solution $(A,B,0)$ to the Vafa-Witten equations,
$$\|B\|_{L^{\infty}(X)}\leq \la\|B\|_{L^{2}(X)}.$$
\end{thm}
\begin{lem}(\cite{FL1} Lemma 6.6)\label{110}
Let $X$ be a closed,\ oriented,\ four-dimensional manifold with Riemannian metric,\ $g$\ Then there are positive constants $c=c(g)$ and $\varepsilon=\varepsilon(g)\in[0,1)$,\ with following significance.\ If $G$ is a compact Lie group,\ $A$ be a connection Sobolev class $L^{2}_{2}(X)$ on a principle $P$ over $X$ with
$$\|F_{A}\|_{L^{2}(X)}\leq\varepsilon,$$
and $\nu\in\Om^{2,+}(X,\mathfrak{g}_{P})$,\ then
\begin{equation}
\|\nu\|_{L_{1,A}^{2}(X)}\leq c\big{(}\|d^{+,\ast}_{A}\nu\|_{L^{2}(X)}+\|\nu\|_{L^{2}(X)}\big{)}.
\end{equation}
\end{lem}
For any real constant $C\in\mathbb{R}^{+}$,\ we defined the $C$-truncated moduli space.
$$M_{VW}^{C}:=\{(A,B,0)\in M_{VW}\mid\|B\|_{L^{2}(X)}\leq C\},$$
then we have a similar result as follow
\begin{thm}\label{5}
Let $X$ be a closed,\ oriented,\ four-dimensional manifold;\ and $P\rightarrow X$ be a principal $G$-bundle with $G$ being a compact Lie group  with $p_{1}(P)$ negative and be such that there exist $\mu,\de>0$ with the property that $\mu(A)\geq\mu$ for all $A\in\mathfrak{B}_{\de}(P,g)$,\ where $\mu(A)$ is as in (\ref{12}).\ There exist a positive constant,\ $C$,\ with the following significance.\ If $(A,B,0)$ is an $L^{2}_{1}$ solution of the Vafa-Witten equations and $\|B\|_{L^{2}(X)}\leq C$,\ then
$$M_{VW}^{C}=\{(A,B,0)\in M_{VW}\mid F^{+}_{A}=0,\ B=0\}.$$
Moreover,\ if $M_{ASD}$ is non-empty and $M_{VW}\backslash M_{ASD}$ is also non-empty,\ then the moduli space $M_{VW}$ is not connected.
\end{thm}
\begin{proof}
From Lemma \ref{110} and the Definition \ref{111} of $\mu(A)$,\ $\forall\nu\in\Om^{2,+}(X,\mathfrak{g}_{P})$,\ we have
\begin{equation}\label{112} \begin{split}
\|\nu\|_{L_{1,A}^{2}(X)}&\leq c\big{(}\|d^{+,\ast}_{A}\nu\|_{L^{2}(X)}+\|\nu\|_{L^{2}(X)}\big{)}\\
&\leq c\big{(}1+1/\sqrt{\mu_{g}(A)}\big{)}\|d^{+,\ast}_{A}\nu\|_{L^{2}(X)}\\
&\leq c\big{(}1+1/\sqrt{\mu/2}\big{)}\|d^{+,\ast}_{A}\nu\|_{L^{2}(X)}.\\
\end{split}
\end{equation}
Since $(A,B,0)$ is a solution of the Vafa-Witten equations,\ then we have
\begin{equation}\nonumber
\begin{split}
\|F^{+}_{A}\|_{L^{2}(X)}&=\frac{1}{8}\|[B.B]\|_{L^{2}(X)}\\
&\leq \frac{1}{4}\|B\|_{L^{4}(X)}^{2}\leq\frac{1}{4}\|B\|_{L^{\infty}(X)}\|B\|_{L^{2}(X)}\\
&\leq \la(X)\|B\|_{L^{2}(X)}^{2}.\\
\end{split}
\end{equation}
We choose $C$ sufficiently small such that $\la C^{2}\leq\varepsilon$,\ where $\varepsilon$ is a as in hypotheses of Lemma \ref{110},\ we apply a priori estimate (\ref{112}) to $\nu=B$ to obtain
\begin{equation}\label{113}
\|B\|_{L^{2}_{1}(X)}\leq c\big{(}1+1/\sqrt{\mu/2}\big{)}\|d^{+,\ast}_{A}B\|_{L^{2}(X)}.
\end{equation}
As $0=d^{\ast}_{A}B=2d^{+,\ast}_{A}B$ on $X$,\ therefore,\ $B=0$ on $X$ by (\ref{113}).

If $M_{ASD}$ is non-empty and $M_{VW}\backslash M_{ASD}$ is also non-empty,\ since the map $(A,B)\mapsto\|B\|_{L^{2}}$ is continuous,\ then the moduli space $M_{VW}$ is not connected.

\end{proof}
\begin{cor}
Assume the hypotheses of Corollary \ref{92} and that $g$ is generic.\ There exist a positive constant,\ $C$,\ with the following significance.\ If $(A,B,0)$ be an $L^{2}_{1}$ solution of the Vafa-Witten equations and $\|B\|_{L^{2}(X)}\leq C$,\ then
$$M_{VW}^{C}=\{(A,B,0)\in M_{VW}\mid F^{+}_{A}=0,\ B=0\}.$$
Moreover,\ if $M_{ASD}$ is non-empty and $M_{VW}\backslash M_{ASD}$ is also non-empty,\ then the moduli space $M_{VW}$ is not connected.
\end{cor}

\subsection*{Acknowledgment}
I would like to thank my supervisor Professor Sen Hu for suggesting me to consider this problem,\ and for providing numerous ideas during the course of
stimulating exchanges.\ I would like to thank the anonymous referee for a careful reading of my manuscript and helpful comments.\ I also would like to thank Siqi He,\ Zhi Hu,\ Ruiran Sun and Yuuji Tanaka for further discussions about this work.\ This research
is partially supported by Wu Wen-Tsun Key Laboratory of Mathematics of Chinese Academy of Sciences at USTC.

\end{document}